\documentclass[12pt]{article}

\usepackage[largesc,osf]{newtxtext}
\usepackage{cite,enumitem,mathtools}
\usepackage[amsthm,nosymbolsc,vvarbb]{newtxmath}
\usepackage[cal=euler,frak=euler]{mathalpha} 
\usepackage{bm} 

\usepackage[margin=2.5cm]{geometry}

\usepackage{tikz}
\usetikzlibrary{calc,cd}
\tikzcdset{arrow style=tikz}
\tikzset{cd/.style=matrix of math nodes}

\title{Algebras of distributions suitable for 
       phase-space quantum \\ mechanics.
       II. Topologies on the Moyal algebra}

\author{Joseph C. Várilly and José M. Gracia-Bondía \\[12pt]
{\small
Escuela de Matemática, Universidad de Costa Rica,
11501 San José, Costa Rica}}

\date{J. Math.\ Phys. \textbf{29} (1988), 880--887}

\newcommand{\Dl}{\Delta}            
\newcommand{\dl}{\delta}            
\newcommand{\eps}{\varepsilon}      
\newcommand{\Ga}{\Gamma}            
\newcommand{\om}{\omega}            
\newcommand{\ze}{\zeta}             

\newcommand{\bC}{\mathbb{C}}
\newcommand{\bN}{\mathbb{N}}
\newcommand{\bR}{\mathbb{R}}
\newcommand{\one}{\mathbb{1}}         

\newcommand{\HS}{\mathcal{HS}}
\newcommand{\sB}{\mathcal{B}}
\newcommand{\sD}{\mathcal{D}}
\newcommand{\sE}{\mathcal{E}}       
\newcommand{\sG}{\mathcal{G}}
\newcommand{\sH}{\mathcal{H}}       
\newcommand{\sI}{\mathcal{I}}
\newcommand{\sK}{\mathcal{K}}
\newcommand{\sL}{\mathcal{L}}
\newcommand{\sM}{\mathcal{M}}       
\newcommand{\sO}{\mathcal{O}}
\newcommand{\sS}{\mathcal{S}}       
\newcommand{\sT}{\mathcal{T}}

\newcommand{\sW}{\mathcal{W}}

\bmdefine{\sss}{s}                  

\renewcommand{\geq}{\geqslant}      
\renewcommand{\leq}{\leqslant}      
\newcommand{\ovl}{\overline}        
\newcommand{\ox}{\otimes}           
\newcommand{\stroke}{\mathbin|}
\newcommand{\wt}{\widetilde}        
\newcommand{\x}{\times}             
\renewcommand{\:}{\colon}           

\newcommand{\Wbar}{\overline{W}}

\newcommand{\hatox}{\mathbin{\widehat\otimes}}

\newcommand{\half}{{\mathchoice{\thalf}{\thalf}{\shalf}{\shalf}}}
\newcommand{\shalf}{{\scriptstyle\frac{1}{2}}} 
\newcommand{\thalf}{\tfrac{1}{2}}   

\newcommand{\hideqed}{\renewcommand{\qed}{}} 

\newcommand{\set}[1]{\{\,#1\,\}}     
\newcommand{\word}[1]{\quad\text{#1}\quad} 

\newcommand{\duo}[2]{\langle#1,#2\rangle} 
\newcommand{\ketbra}[2]{\lvert#1\rangle\langle#2\rvert}
\newcommand{\pd}[2]{\frac{\partial#1}{\partial#2}} 
\newcommand{\scal}[2]{(#1\stroke#2)} 

\theoremstyle{plain}
\newtheorem{thm}{Theorem}             
\newtheorem{lema}{Lemma}              

\theoremstyle{definition}

\theoremstyle{remark}
\newtheorem{remk}{Remark}             

\makeatletter
\renewcommand{\section}{\@startsection{section}{1}{\z@}%
							 {-3.25ex \@plus -1ex \@minus -.2ex}%
							 {1.5ex \@plus.2ex}%
							 {\normalfont\large\bfseries}}
\renewcommand{\subsection}{\@startsection{subsection}{2}{\z@}%
							 {-3.25ex \@plus -1ex \@minus -.2ex}%
							 {1.5ex \@plus .2ex}%
							 {\normalfont\normalsize\bfseries}}
\makeatother

\begin{document}

\maketitle

\begin{abstract}
The topology of the Moyal $*$-algebra may be defined in three ways:
the algebra may be regarded as an operator algebra over the space of
smooth declining functions either on the configuration space or on the
phase space itself; or one may construct the $*$-algebra~via a
filtration of Hilbert spaces (or other Banach spaces) of
distributions. We prove the equivalence of the three topologies
thereby obtained. As a consequence, by filtrating the space of
tempered distributions by Banach subspaces, we give new sufficient
conditions for a phase-space function to correspond to a trace-class
operator via the Weyl correspondence rule.
\end{abstract}

\section{Introduction} 

In the previous article \cite{Phobos} (hereinafter referred to simply
as~\textbf{I}), we laid the foundations of a promising mathematical
mold for the phase-space formulation of quantum mechanics. In this
paper we obtain some less straightforward results, in keeping with the
preliminary stage of the program outlined in~\textbf{I}.

A wealth of information may be gained by characterizing the Moyal
$*$-algebra $\sM$ as a suitably defined limit of a family of Banach
spaces, which form a filtration of the space of tempered distributions
on phase-space. In fact, we introduce several variants of this
filtration, depending on whether to use Hilbert algebras or some other
kind of Banach algebras.

The definition of an adequate topology on $\sM$ is obviously of great
importance. For physical reasons, we should in principle consider two
topologies on $\sM$, depending on whether we wish to link the theory
with ordinary quantum mechanics, or to study the dynamics in $\sM$ in
its natural context. A third topology on $\sM$ is given by the
filtration. We will show that these three topologies are equivalent.
	
The paper is organized as follows. In Sec.~\ref{sec:kernel-product} we
review the twisted product and the integral transformation of Wigner,
which intertwines the twisted product with the composition of kernel
functions. We show how this transformation and the kernel theorem
establish a link between $\sM$ and the algebra $\sL_b(\sS_1)$. We also
sketch how the Wigner transformation and the twisted product yield a
constructive proof of the Stone--von~Neumann theorem. In
Sec.~\ref{sec:operator-topologies} we introduce two topologies
on~$\sM$, regarding $\sM$ first as an operator algebra over a space of
test functions on configuration space, and secondly as an operator
algebra over functions on phase space; and we prove the equivalence of
these two topologies. In Sec.~\ref{sec:many-filtrations} we introduce
the filtrations of the space of tempered distributions (on phase
space) and characterize $\sM$ in terms of these filtrations. From this
characterization, we obtain the third topology on~$\sM$, and we prove
its equivalence with the previous two. In Sec.~\ref{sec:trace-class},
we obtain conditions which imply that certain functions on phase space
correspond to trace-class operators in the usual formulation of
quantum mechanics.

\section{Twisted products and the kernel theorem} 
\label{sec:kernel-product}

In the usual approach to phase-space quantum mechanics, via the Weyl
correspondence between functions and operators
\cite{AmietH81,Anderson72,Pool66}, $L^2$~functions correspond to
Hilbert--Schmidt operators; since these have $L^2$~kernels, we may
relate the twisted product to the composition of integral kernels by
some transformation of $L^2(\bR^{2N})$ onto itself. This transformation
turns out to be the prescription introduced by
Wigner~\cite{Wigner32,Phoebe} to associate a ``distribution function''
on phase space to a Schr\"odinger wave function. Moreover, it maps the
twisted Hermite basis of~\textbf{I} onto the ordinary Hermite basis
for $L^2(\bR^{2N})$. Also, we can build on the observation by
Cressman~\cite{Cressman76} that the Wigner transformation allows us to
transfer the kernel theorem to the twisted product calculus, and in
this way we identify the Moyal $*$-algebra $\sM$ in terms of more
familiar spaces.

We use the same notations as~\textbf{I}. Also, we write
$\sS_1 = \sS(\bR)$ and $\sS_2 = \sS(\bR^2)$ to denote the spaces of 
rapidly decreasing smooth functions on $\bR$ and~$\bR^2$, and 
$\sS'_1$, $\sS'_2$ for their dual spaces of tempered distributions. 
When $\phi,\psi \in \sS_1$, we define $\phi \ox \psi \in \sS_2$ by 
$(\phi \ox \psi)(q,p) := \phi(q)\psi(p)$. If $E,F$ are locally convex
spaces, $\sL(E,F)$ will denote the space of continuous linear maps
$\: E \to F$, which we abbreviate to $\sL(E)$ in the case $E = F$.

If $f,g \in \sS_2$ [or if $f,g \in L^2(\bR^2)$], we write
$$
(f \circ g)(x,y) := \frac{1}{\sqrt{4\pi}} \int_\bR f(x,z) g(z,y) \,dz
$$
which is the ``kernel product'' of $f$ and~$g$. We also introduce the
maps $R$, $\Phi$, $W$ from $\sS_2$ onto $\sS_2$ by
\begin{align*}
(Rf)(x,y) 
&:= \biggl( \frac{x+y}{\sqrt{2}}, \frac{x-y}{\sqrt{2}} \biggr)
\\
(\Phi f)(x,y) &:= \frac{1}{\sqrt{2\pi}} \int_\bR f(x,z) e^{-iyz} \,dz
\\
(Wf)(x,y) &:= (R\Phi^{-1}f)(x,y) = \frac{1}{\sqrt{2\pi}} \int_\bR
f\biggl( \frac{x+y}{\sqrt{2}}, z \biggr) \,e^{i(x-y)z/\sqrt{2}} \,dz.
\end{align*}
Clearly $R$ and the ``partial Fourier transform'' $\Phi$, and hence 
$W$, are Fréchet-space isomorphisms of $\sS_2$ onto itself, and 
extend to unitary operators on $L^2(\bR^2)$. We call $W$ the
\textit{Wigner transformation} on $\sS_2$. It intertwines the twisted 
product and the kernel product on $\sS_2$:
\begin{equation}
W(f \x g) = Wf \circ Wg.
\label{eq:inter-twine} 
\end{equation}
Indeed, since $W^{-1} = \Phi R^{-1} = \Phi R$, a straightforward
computation shows that $W^{-1}(Wf \circ Wg) = f \x g$ for
$f,g \in \sS_2$. The identity~\eqref{eq:inter-twine} gives the
connection between the Weyl operator formalism and the twisted product
calculus.

Indeed, more is true: if the set $\set{f_{mn} : m,n \in \bN}$ denotes 
the ``twisted Hermite basis'' of $\sS_2$ discussed in~\textbf{I}:
\begin{equation}
f_{mn}(q,p) = 2(-1)^n \sqrt{\frac{n!}{m!}} (q-ip)^{m-n} 
L_n^{m-n}(q^2+p^2) e^{-(q^2+p^2)/2}
\label{eq:matrix-basis} 
\end{equation}
(if $m \geq n$; $f_{mn} := f_{nm}^*$ otherwise); and if $h_m \in \sS_1$
denotes the Hermite function
$$
h_m(x) := \frac{1}{\sqrt{2^{m-1}m!}} H_m(x) e^{-x^2/2},
$$
then a direct calculation shows that
\begin{equation}
W(f_{mn}) = h_m \ox h_n.
\label{eq:bases-matching} 
\end{equation}

We may extend $W$ to $\sS'_2$ by duality in the usual way. Indeed,
writing $\Wbar := R\Phi$, we find
$$
\duo{Wf}{g} = \duo{R\Phi^{-1}f}{g} = \duo{\Phi^{-1}f}{Rg}
= \duo{f}{\Phi^{-1}Rg} = \duo{f}{\Wbar^{-1}g} \word{for} f,g\in \sS_2,
$$
so we define $WT$ for $T \in \sS'_2$ by
$$
\duo{WT}{h} := \duo{T}{\Wbar^{-1}h}  \word{for}  h \in \sS_2.
$$
Similarly, we may define $\duo{\Wbar T}{h} := \duo{T}{W^{-1}h}$. Since
$\Wbar^{-1} \: \sS_2 \to \sS_2$ is a topological isomorphism, so is
its transpose $W \: \sS'_2 \to \sS'_2$ (where $\sS'_2$ carries the
strong dual topology~\cite{Horvath66,Treves67}).

Now $\sS_2$ acts on $\sS_1$ by
$$
(f \cdot \phi)(x) 
:= \frac{1}{\sqrt{4\pi}} \int_\bR f(x,y) \phi(y) \,dy.
$$
Note that
$\duo{f \cdot \phi}{\psi} = \frac{1}{\sqrt{2}} \duo{f}{\psi \ox \phi}$
for $f \in \sS_2$, $\phi,\psi \in \sS_1$. Thus we may define 
$T \cdot \phi$, for $T \in \sS'_2$, $\phi \in \sS_1$, by
transposition:
$$
\duo{T \cdot \phi}{\psi} := \frac{1}{\sqrt{2}} \duo{T}{\psi \ox \phi}.
$$

\begin{remk} 
We also observe that
\begin{equation}
\scal{\phi}{WT \cdot \psi}
= \frac{1}{2} \duo{T}{\Wbar^{-1}(\phi^* \ox \psi)}. 
\label{eq:Wigner-transform} 
\end{equation}
Indeed,
$$
\scal{\phi}{WT \cdot \psi}
= \frac{1}{\sqrt{2}} \duo{\phi^*}{WT \cdot \psi}
= \frac{1}{2} \duo{WT}{\phi^* \ox \psi} 
= \frac{1}{2} \duo{T}{\Wbar^{-1}(\phi^* \ox \psi)}.
$$
The identity \eqref{eq:Wigner-transform} is Moyal's connection between
the Weyl operator formalism and the calculus of Wigner functions. We
may interpret it thus: to calculate the transition probabilities for
an observable $T$ between pure states represented state vectors
$\psi,\phi$, one may compute the scalar product of the operator
$WT\cdot$ with the ket $|\psi\rangle$ and the bra $\langle\phi|$ (in
Dirac's terminology), or equivalently one may take the expected value
of $T$ with respect to the Wigner ``distribution function''
$$
\Wbar^{-1}(\phi^* \ox \psi)
= \int_\bR \phi^* \biggl( \frac{q+t}{\sqrt{2}} \biggr) \psi
\biggl( \frac{q-t}{\sqrt{2}} \biggr) e^{ipt} \,dt.
$$
\end{remk}

Furthermore, if $f,g,h \in \sS_2$ and if $\tilde f(q,p) := f(p,q)$,
then
$$
\duo{f\circ g}{h} = \duo{f}{h \circ \tilde g}
= \duo{g}{\tilde f \circ h} 
= \frac{1}{\sqrt{2}} \iiint f(q,t) g(t,p) h(q,p) \,dt \,dq \,dp
$$
so $\duo{T \circ f}{h} := \duo{T}{h \circ \tilde f}$ defines
$T \circ f$ by transposition, for $T \in \sS'_2$, $f \in \sS_2$.

\begin{lema} 
\label{lm:Wigner-homom}
If $T \in \sS'_2$, $f \in \sS_2$, and $\phi,\psi \in \sS_1$, then
\begin{enumerate}
\item[\textup{(i)}]
$W(T \x f) = WT \circ Wf$; 
\item[\textup{(ii)}]
$T \circ (\phi \ox \psi) = (T \cdot \phi) \ox \psi$.
\end{enumerate}
\end{lema}

\begin{proof}
Since $W\Wbar^{-1}h = R\Phi^{-2}Rh = \tilde h$, we get
\begin{align*}
\duo{W(T \x f)}{h} 
&= \duo{T}{f \x \Wbar^{-1}h} = \duo{\Wbar T}{W(f \x \Wbar^{-1}h)}
\\
&= \duo{\Wbar T}{Wf \circ \tilde h} 
= \duo{\wt{WT}}{Wf \circ \tilde h}
\\
&= \duo{WT}{h \circ \wt{Wf}} = \duo{WT \circ Wf}{h}
\end{align*}
for $h \in \sS_2$. Also,
\begin{align*}
\duo{T \circ (\phi\ox\psi)}{h}
&= \duo{T}{h \circ (\psi\ox\phi)} = \duo{T}{(h\cdot\psi) \ox \phi}
\\
&= \frac{1}{\sqrt{2}} \duo{T\cdot\phi}{h\cdot\psi} 
= \duo{(T\cdot\phi) \ox \psi}{h}.
\tag*{\qed}
\end{align*}
\hideqed
\end{proof}

Writing $(ZT)(\phi) := T \cdot \phi$, the kernel theorem for tempered
distributions~\cite{Treves67} states that $Z$ is an isomorphism from
$\sS'_2$ onto $\sL(\sS_1,\sS'_1)$. We now have the following theorem.

\begin{thm} 
\label{th:Wigner-homom}
$ZW(\sM_L) = \sL(\sS_1)$; moreover, $\sM_L$ is an algebra under the
twisted product, and $ZW \: \sM_L\to \sL(\sS_1)$ is an algebra
isomorphism.
\end{thm}

\begin{proof} 
$T \in \sM_L$ iff $T \x f \in \sS_2$ for all $f \in \sS_2$, iff
$T \x f \in \sS_2$ for all $f$ of the form $W^{-1}(\phi \ox \psi)$ 
with $\phi,\psi \in \sS_1$, since $\sS_1 \ox \sS_1$ is dense in
$\sS_2$ and $f \mapsto T \x f$ is continuous (see~\textbf{I}). Thus
$T \in \sM_L$ iff 
$$
W(T \x W^{-1}(\phi \ox \psi)) = WT \circ (\phi \ox \psi) 
= (WT \cdot \phi) \ox \psi = ZWT(\phi) \ox \psi \in \sS_2
$$
for all $\phi,\psi \in \sS_1$, iff $ZWT(\phi) \in \sS_1$ for all 
$\phi \in \sS_1$. This last statement holds since the reduction map 
$\chi \ox \psi \mapsto \duo{R}{\psi} \chi$, for any $R \in \sS'_1$, 
extends by linearity and continuity to the completed projective 
tensor product $\sS_1 \hatox \sS_1 \simeq \sS_2$, and hence maps 
$\sS_2$ into $\sS_1$ continuously.

Taking $T$, $f$ as before, and $\chi \in \sS_1$, we have
\begin{align*}
ZW(T \x f)(\chi) 
&= Z(WT \circ Wf)(\chi) = Z(WT \circ (\phi \ox \psi))(\chi)
\\
&= Z((WT \cdot \phi) \ox \psi)(\chi)
= ((WT \cdot \phi) \ox \psi) \cdot \chi
\\
&= WT \cdot \phi \duo{\psi}{\chi} = WT \cdot (Z(\phi \ox \psi)(\chi))
\\
&= ZWT(ZWf(\chi))
\end{align*}
and [by density of $W(\sS_1 \ox \sS_1)$ in $\sS_2$], we get 
$ZW(T \x f) = ZW(T) ZW(f)$ for any $f \in \sS_2$, $T \in \sS'_2$. If
$S \in \sM_L$, we then have
\begin{align*}
\duo{ZW(T \x S)(\phi)}{\psi} 
&= \duo{W(T \x S) \cdot \phi}{\psi}
= \frac{1}{\sqrt{2}} \duo{W(T \x S)}{\psi \ox \phi}
\\
&= \frac{1}{\sqrt{2}} \duo{T}{S \x \Wbar^{-1}(\psi \ox \phi)}
= \frac{1}{\sqrt{2}} \duo{T}{S \x W^{-1}(\phi \ox \psi)} 
\\
&= \frac{1}{\sqrt{2}} \duo{\Wbar T}{WS \circ (\phi \ox \psi)}
= \frac{1}{\sqrt{2}} \duo{\Wbar T}{ZWS(\phi) \ox \psi}
\\
&= \frac{1}{\sqrt{2}} \duo{WT}{\psi \ox ZWS(\phi)}
= \duo{ZWT(ZWS(\phi))}{\psi}
\end{align*}
for any $\psi \in \sS_1$. In particular, if $R \in \sM_L$, we get 
$ZW(R \x S) = ZW(R)ZW(S) \in \sL(\sS_1)$ and since $ZW$ is one-to-one
$\: \sS'_2 \to \sL(\sS_1,\sS'_1)$, we conclude that
$R \x S \in \sM_L$, so that $\sM_L$ is in fact an algebra, and
$ZW \: \sM_L \to \sL(\sS_1)$ is a bijective homomorphism.
\end{proof}

\begin{remk} 
By analogous arguments, or more directly by noting that 
$(WT)^* = \Wbar(T^*)$, one can show that
$Z\Wbar \: \sM_R \to \sL(\sS_1)$ is an algebra isomorphism.
\end{remk}

Before proceeding, we observe that the Wigner transformation gives a
direct, constructive proof of the Stone--von~Neumann theorem. We
regard $(\sS_2,\x)$ as an operator algebra and look at its left
regular representation $\pi(f)g := f \x g$. If
$f_0(u) := 2e^{-u^2/2}$, define $\om \: \sS_2 \to \bC$ by
$\om(g) := \scal{f_0}{g \x f_0}$. This $\om$ is linear and continuous
on~$\sS_2$, and $\om(g^* \x g) = \|g \x f_0\| \geq 0$. Since the
Gaussian function $f_0$ has the property that
$f_0 \x g \x f_0 = \scal{f_0}{g} f_0$, we have 
$\om(g) := \scal{f_0}{g}$.

Using the positive functional $\om$, we can apply the
Gelfand--Na\u{\i}mark--Segal construction to~$\sS_2$. We observe that
$\sK := \set{g \in \sS_2 : g \x f_0 = g}$ and
$\sK_0 := \set{g \in \sS_2 : g \x f_0 = 0}$ are closed left ideals in
$\sS_2$, that if $\eta \: \sS_2 \to \sS_2/\sK_0$ is the canonical
projection, then $\sS_2/\sK_0$ becomes a prehilbert space with inner
product
$$
\scal{\eta(g)}{\eta(f)} := \om(g^*\x f) = \scal{g \x f_0}{f \x f_0},
$$
whose completion is denoted $\sH_\om$, and that
$\eta(g) \mapsto \eta(f \x g)$ extends to an operator
$\pi_\om(f) \in \sL(\sH_\om)$ with $\|\pi_\om(f)\| \leq \|f\|$.

It is easily verified that $\eta(f_{mn}) = 0$ in $\sH_\om$ if
$n \neq 0$, and that the set $\set{\eta(f_{m0}) : m \in \bN}$ is an
orthonormal basis for $\sH_\om$. Now
$\sum_{m,n=0}^\infty c_{mn}f_{mn}$ lies in~$\sK$ iff
$c \in \sss$ (see~\textbf{I}) and $c_{mn} = 0$ for $n \neq 0$.
Thus $\eta \: \sK \to \sH_\om$ is isometric if $\sK$ is given the norm
of $L^2(\bR^2)$, and so extends to a unitary map
$V \: \ovl{\sK} \to \sH_\om$, where $\ovl{\sK}$ is the $L^2$-closure
of~$\sK$.

If $g = \sum_{m,n=0}^\infty c_{mn} f_{mn} \in \sS_2$ and
$\pi_\om(g) = 0$, then for all~$n$ we have
$$
0 = \eta(g \x f_{n0}) = \sum_{m=0}^\infty c_{mn}\eta(f_{m0}),
$$
so that all $c_{mn} = 0$: hence $\pi_\om \: \sS_2 \to \sL(\sH_\om)$ is
a faithful representation of~$\sS_2$. Using Schur's lemma, we can
check that $\pi_\om$ is irreducible.

Define the projector $P \: L^2(\bR^2) \to L^2(\bR)$ by
$P(\phi \ox \psi) := \scal{h_0}{\psi} \phi$ and
$Q \: L^2(\bR) \to L^2(\bR) \ox \bC h_0$ by $Q(\phi) := \phi \ox h_0$.
Then $PWV^{-1} \: \sH_\om \to L^2(\bR)$ is unitary, with inverse
$VW^{-1}Q$, and $PWV^{-1}(\eta(f_{m0})) = h_m$. Then we may calculate
that
$$
PW(q \x f_{m0}) = \sqrt{m}\,h_{m-1} + \sqrt{m + 1}\,h_{m+1}
= \sqrt{2}\,q h_m,
$$
and similarly $PW(p \x f_{m0}) = -\sqrt{2}\,i\,dh_m/dq$, so that 
$$
PW(q \x f) = \sqrt{2}\,q PWf, \quad 
PW(p \x f) = -\sqrt{2}\,i\,\frac{d}{dq} PWf  \word{for} f \in \sK.
$$

Let us write $\pi_s(f) := PWV^{-1}\pi_\om(f)VW^{-1}Q$ for
$f \in \sS_2$; we may call $\pi_s$ the \textit{Schr\"odinger
representation} of~$\sS_2$ on $L^2(\bR)$. This brings us to the
Stone--von~Neumann theorem, which in the present context states that
any representation of the twisted product algebra $\sS_2$ is
equivalent to a multiple of $\pi_s$; we show this to be true for the
left regular representation $\pi$, as a simple consequence of the GNS
construction. The unitary equivalence which decomposes $\pi$ is just
the Wigner transformation.

\begin{thm} 
\label{th:infinite-multy}
$W\pi(f)W^{-1} = \pi_s(f) \ox \one$, for all $f \in \sS_2$.
\end{thm}

\begin{proof}
{}From Lemma~\ref{lm:Wigner-homom}, for $\phi,\psi \in \sS_1$ we
obtain
\begin{align*}
W\pi(f)W^{-1}(\phi \ox \psi)
&= W(f \x W^{-1}(\phi \ox \psi)) 
\\
&= Wf \circ (\phi \ox \psi)) = (Wf \cdot \phi) \ox \psi
\end{align*}
and also
\begin{align*}
\pi_s(f)\phi &= PWV^{-1} \pi_\om(f) VW^{-1}(\phi \ox h_0)
\\
&= PWV^{-1}V(f \x W^{-1}(\phi \ox h_0)) = Wf \cdot \phi
\end{align*}
for all $\phi,\psi \in \sS_1$. Since the functions $\phi \ox \psi$
generate a dense subspace of $L^2(\bR^2)$, we are done.
\end{proof}

\begin{remk} 
In the previous discussion, we may replace $L^2(\bR)$ by $\sS_1$,
$L^2(\bR^2)$ by $\sS_2$, and $\sS_2$ by $\sM_L$. Then $PWV^{-1}$ is a
Fréchet-space isomorphism from $\sS_2/\sK_0$ onto $\sS_1$, and so we
can extend the representations $\pi_\om$ and~$\pi_s$ to~$\sM_L$
(acting on $\sS_2/\sK_0$ and $\sS_1$, respectively). With $\pi$ now
denoting the left regular representation of $\sM_L$ on its left ideal
$\sS_2$, Theorem~\ref{th:infinite-multy} remains valid, with the added
advantage that we can write $\pi_s(q)\phi = \sqrt{2}\,q\phi$,
$\pi_s(p)\phi = -\sqrt{2}\,i\,d\phi/dq$ for $\phi \in \sS_1$, thus
displaying the Schr\"odinger representation $\pi_s$ in its familiar
form, modulo a normalization factor.
\end{remk}

\begin{remk} 
{}From the proof, we see that
\begin{equation}
(\pi_s(f)\phi)(x) = (Wf \cdot \phi)(x)
= \frac{1}{2\pi\sqrt{2}} \iint_{\bR^2} f\biggl(
\frac{x+y}{\sqrt{2}},z \biggr) e^{i(x-y)z/\sqrt{2}} \phi(y) \,dz\,dy.
\label{eq:Weyl-Psido} 
\end{equation}
Thus the operator $\pi_s(f)$ is the pseudodifferential operator (in
the sense of H\"ormander~\cite{Hormander79}) associated to the
``symbol''~$f$. This is precisely Weyl's quantization rule in more
fashionable language. If $f = p^2/2m + V(q)$, we get easily from
\eqref{eq:Weyl-Psido}, at least formally, the ``Schr\"odinger
operator'': $\pi_s(f) = -(2/m)\,\Dl + V(q)$.
\end{remk}

\section{Operator topologies on $\sM$} 
\label{sec:operator-topologies}

Now $\sL(\sS_1)$ carries a natural topology, that of uniform
convergence on bounded subsets of $\sS_1$ [under which it is a
complete, nuclear, reflexive locally convex
space~\cite{Grothendieck55,Vogt84}, usually denoted by
$\sL_b(\sS_1)$]. This is the standard example of a locally convex
algebra which is neither Fréchet nor DF. Let $\sT_1$ be the unique
locally convex topology on $\sM_L$ so that
$ZW \: (\sM_L,\sT_1) \to \sL_b(\sS_1)$ is a homeomorphism.

A second method of topologizing $\sM_L$ is as follows. If 
$(\eps_u \tau_u f)(v) := e^{iu'Jv}f(v-u)$, we have shown in~\textbf{I}
that $\eps_u \tau_u(T \x f) = T \x (\eps_u \tau_u f)$ for all
$T \in \sS'_2$, $f \in \sS_2$, $u \in \bR^2$; and moreover, that if
$L \: \sS_2 \to \sE(\bR^2)$ is a continuous linear map commuting with
every $\eps_u \tau_u$, then $L(f) = T \x f$ for all $f \in \sS_2$,
for some $T \in \sS'_2$. Thus we find, writing $L_T(f) := T \x f$,
that
$$
\set{L_S : S \in \sM_L} = \set{L \in \sL_b(\sS_2)
: L\,\eps_u \tau_u = \eps_u \tau_u \,L \text{ for all } u \in \bR^2}.
$$
Let $\sT_2$ denote the topology on $\sM_L$ so that $S \mapsto L_S$ is a 
homeomorphism of $(\sM_L,\sT_2)$ onto this closed subspace of 
$\sL_b(\sS_2)$, where the subscript $b$ denotes the topology of 
uniform convergence on bounded subsets of~$\sS_2$. Then, we have the 
following theorem.

\begin{thm} 
\label{th:same-topologies}
The topologies $\sT_1$ and $\sT_2$ on $\sM_L$ coincide.
\end{thm}

We next define the topology of $\sM_R$ so that $S \mapsto S^*$ is a
topological isomorphism of $\sM_L$ onto~$\sM_R$. Finally, we give
$\sM$ the natural topology of the intersection $\sM_L \cap \sM_R$,
that is, the weakest locally convex topology so that both inclusions
$\sM \subset \sM_L$, $\sM \subset \sM_R$ are continuous. From the
definition of $\sT_1$, it is already known that $\sM$ is a complete,
nuclear, semireflexive locally convex $*$-algebra with a separately
continuous multiplication and a continuous involution; on the other
hand, via $\sT_2$, $\sM$ may be regarded as an operator $*$-algebra on
$\sS_2$. This is particularly useful with a view to solving
Schr\"odinger equations of the form
$$
2i\,\pd{U}{t} = H \x U(t); \quad U(0) = \one,
$$
where $H \in \sM$, $U(t) \in \sM$ for $t \in \bR$, using operator
semigroup theory. Note that the semigroup property gives 
$U(s + t) = U(s)U(t)$. In this formula $H$ denotes a time-independent
Hamiltonian and $U(t)$ is the ``twisted exponential''
\cite{BayenM82,Huguenin78} or ``evolution function'' \cite{HAtom}
associated with this Hamiltonian, which contains the dynamical
information for the system analogously with the Green's function in
the conventional formulation.

\begin{proof}[Proof of Theorem~\ref{th:same-topologies}]
For $i = 1,2$, let $\sT'_i$ be the locally convex topology on
$W(\sM_L)$ so that $W \: (\sM_L,\sT_i) \to (W(\sM_L),\sT'_i)$ is a
homeomorphism. It suffices to show that $\sT'_1 = \sT'_2$.

A basic neighbourhood of~$0$ for $\sT'_1$ is of the form
$$
(A;V)_1 := 
\set{T \in W(\sM_L) : T\cdot\phi \in V \text{ for all } \phi \in A}
$$
where $A$ is a bounded subset of $\sS_1$ and $V$ is a
zero-neighbourhood in $\sS_1$. A basic neighbourhood of~$0$ for
$\sT'_2$ is of the form
$$
(B;U)_2 :=
\set{T \in W(\sM_L) : T \circ f \in U \text{ for all } f \in B}
$$
where $B$ is a bounded subset of $\sS_2$ and $U$ is a 
zero-neighbourhood in~$\sS_2$. Since $\sS_2 = \sS_1 \hatox \sS_1$
\cite{Treves67}, it suffices~\cite{Schaefer66} to consider $B$ of the
form $B = A_1 \ox A_2$, where $A_1$, $A_2$ are bounded subsets
of~$\sS_1$, and $U$ of the form $U = \Ga(V_1 \ox V_2)$, the closed
absolutely convex hull of $V_1 \ox V_2$, where $V_1$, $V_2$ are
zero-neighbourhoods in $\sS_1$. Since $A_2$ is bounded, we can find
$r > 0$ so that $A_2 \subset rV_2$; since
$r^{-1}V_1 \ox rV_2 = V_1 \ox V_2$, we can also assume that
$A_2 \subset V_2$.

For a given such $(B;U)_2$, let $T \in (A_1;V_1)_1$, 
$\phi \in A_1$, $\psi \in A_2$; then from Lemma~\ref{lm:Wigner-homom}
we obtain
$$
T \circ (\phi\ox\psi) = (T\cdot\phi) \ox \psi \in V_1 \ox A_2
\subset V_1 \ox V_2 \subset U
$$
so that $(B;U)_2$ contains $(A_1;V_1)_1$.

On the other hand, let $(A;V)_1$ be given, with $V$ absolutely convex.
Choose $\psi \in \sS_1$, $R \in \sS'_1$ such that $\duo{R}{\psi} = 1$,
and set $V_2 := \set{\phi \in \sS_1 : |\duo{R}{\phi}| \leq 1}$. If
$T \in (A \ox \{\psi\}; \Ga(V \ox V_2))_2$ and $\phi \in A$, then
$T \circ (\phi \ox \psi) = (T \cdot \phi) \ox \psi$ lies in
$\Ga(V \ox V_2)$, and on applying the reduction map
$\chi \ox \om \mapsto \duo{R}{\om} \chi$, we find that
$T \cdot \phi \in V$. Thus $(A;V)_1$ contains a set of the form
$(B;U)_2$.

We have shown that the zero-neighbourhood bases for $\sT'_1$, $\sT'_2$
are equivalent, so $\sT'_1 = \sT'_2$ and hence $\sT_1 = \sT_2$.
\end{proof}

\section{Filtrations of $\sS'_2$} 
\label{sec:many-filtrations}

In this section we introduce several filtrations of $\sS'_2$, in terms
of which $\sM$ may be characterized. We start with a rigged Hilbert
space structure which may be defined (in two-dimensional phase space)
by a two-parameter family of Hilbert spaces. (See~\textbf{I}, Sec.~5.)
In~\textbf{I} we have shown that the basis functions $f_{mn}$
of~\eqref{eq:matrix-basis} are orthonormal with respect to the measure
$(4\pi)^{-1} \,dq\,dp$, that $f_{mn} \x f_{kl} = \dl_{nk} f_{ml}$, and
that
\begin{equation}
H \x f_{mn} = (2m + 1) f_{mn},  \qquad  f_{mn} \x H = (2n + 1) f_{mn}.
\label{eq:eigen-basis} 
\end{equation}

We introduced the Hilbert spaces $\sG_{s,t}$ (for $s,t \in \bR$)
in~\textbf{I} as the completions of $\sS_2$ with respect to the norm
$\|\cdot\|_{st}$, where
\begin{equation}
\biggl\| \sum_{m,n=0}^\infty c_{mn}f_{mn} \biggr\|_{st}^2 
:= \sum_{m,n=0}^\infty (2m + 1)^s (2n + 1)^t |c_{mn}|^2,
\label{eq:many-norms} 
\end{equation}
and we observed that
\begin{equation}
\sS_2 = \bigcap_{s,t\in\bR} \sG_{s,t}, \qquad 
\sS'_2 = \bigcup_{s,t\in\bR} \sG_{s,t}
\label{eq:double-scale} 
\end{equation}
topologically.

Recalling \eqref{eq:bases-matching} that $W(f_{mn}) = h_m \ox h_n$,
we get $ZW(f_{mn}) = \ketbra{h_m}{h_n}$, which shows that 
$ZW(\sG_{0,0}) = \HS(L^2(\bR))$, the Hilbert--Schmidt operators on
$L^2(\bR)$, on account of~\eqref{eq:many-norms}. Next we note that
\begin{align*}
ZWH(h_m) \ox h_n &= (WH\cdot h_m) \ox h_n = WH \circ (h_m \ox h_n)
\\
&= W(H \x f_{mn}) = (2m + 1) W(f_{mn}) = (2m + 1) h_m \ox h_n
\end{align*}
for all $m,n = 0,1,2,\dots$; thus $ZWH(h_m) = (2m + 1) h_m$. This
shows that $(ZWH)^{-1}$ exists and lies in $\HS(L^2(\bR))$.

\begin{thm} 
\label{th:odd-corners}
$\sM_L = \bigcap_{s\in\bR} \bigcup_{t\in\bR} \sG_{s,t}; \quad
\sM_R = \bigcap_{t\in\bR} \bigcup_{s\in\bR} \sG_{s,t}$.
\end{thm}

\begin{proof}
We observe that $\sH_k := \sD((ZWH)^k)$ is a Hilbert space under the
norm $\|\cdot\|_k$, where
$$
\biggl\| \sum_{m=0}^\infty a_m h_m \biggr\|_k^2
:= \sum_{m=0}^\infty (2m + 1)^{2k} |a_m|^2
$$
and that $\sS_1 = \bigcap_{k=0}^\infty \sH_k$ topologically.

We now notice that $T \in \sM_L$ iff $ZWT \in \sL(\sS_1)$ iff for 
all $m \geq 0$, there exists $n \geq 1$ with 
$ZWT \in \sL(\sH_{n-1}, \sH_m)$ or equivalently:
\begin{gather*}
A := (ZWH)^m (ZWT) (ZWH)^{-n+1} \in \sL(L^2(\bR)),
\\
(ZWH)^m (ZWT) (ZWH)^{-n} = A(ZWH)^{-1} \in \HS(L^2(\bR)),
\\
H^{\x m} \x T \x H^{\x(-n)} \in L^2(\bR^2) = \sG_{0,0},  \qquad
T \in \sG_{2m,-2n},
\end{gather*}
iff
$$
T \in \bigcap_{m=0}^\infty \bigcup_{n=0}^\infty \sG_{2m,-2n}
= \bigcap_{s\in\bR} \bigcup_{t\in\bR} \sG_{s,t}.
$$

Note further that $S \in \sM_R$ iff $S^* \in \sM_L$ and
$S^* \in \sG_{t,s}$ iff $S \in \sG_{s,t}$, so that 
$\sM_R = \bigcap_{t\in\bR} \bigcup_{s\in\bR} \sG_{s,t}$ as claimed.
\end{proof}

\begin{remk} 
The observation that $A(ZWH)^{-1}$ is Hilbert--Schmidt for all~$m$ and
some~$n$ is due to J. Unterberger \cite{JUnterberger84} who gives it
in a slightly different form. Similar ideas lie behind the fundamental
approach to the generalized eigenvalue problem by van~Eijndhoven and
de~Graaf~\cite{EijndhovenDG85}.)
\end{remk}

We may visualize the conclusions of Theorem~\ref{th:odd-corners} by
means of the following diagram:
$$
\begin{tikzcd}
\sM_L \rar & \sG_{s,-\infty} \rar & \sG_{s-1,-\infty} \rar & \sS'_2 
\\
\sM_L' \uar & \sG_{s,t-1} \rar \uar
& \sG_{s-1,t-1} \rar \uar & \sG_{-\infty,t-1} \uar
\\
\sG_{+\infty,t} \rar \uar & \sG_{s,t}  \uar \rar
& \sG_{s-1,t} \uar \rar & \sG_{-\infty,t} \uar 
\\
\sS_2 \rar \uar & \sG_{s,+\infty} \rar \uar
& \sM_R' \rar & \sM_R \uar
\end{tikzcd}
$$

Here we set $\sG_{s,-\infty} := \bigcup_{t\in\bR} \sG_{s,t}$, 
$\sG_{-\infty,t} := \bigcup_{s\in\bR} \sG_{s,t}$; then 
$\sM_L = \bigcap_{s\in\bR} \sG_{s,-\infty}$, 
$\sM_R = \bigcap_{t\in\bR} \sG_{-\infty,t}$. This gives us our third 
method of topologizing $\sM_L$: let $\sG_{s,-\infty}$ have the
strongest locally convex topology such that all the inclusions
$\sG_{s,t} \subset \sG_{s,-\infty}$ are continuous. Let $\sT_3$ denote
the weakest locally convex topology on $\sM_L$ such that all the 
inclusions $\sM_L \subset \sG_{s,-\infty}$ are continuous. On the 
diagram, reflection in the principal diagonal [i.e., 
$(s,t) \mapsto (t,s)$] represents complex conjugation, and so the 
topology of $\sM_R$ is defined in the analogous way. Thus, every arrow 
in the diagram represents a continuous (and dense) inclusion.

[We include the left column and bottom row for completeness, although
they are proved in a separate article~\cite{Charon}. We write 
$\sG_{+\infty,t} := \bigcap_{s\in\bR} \sG_{s,t}$ (topologically) and
in~\cite{Charon} we have shown that the dual space
$\sM_L' = \bigcup_{t\in\bR} \sG_{+\infty,t}$ (topologically) and that
$\sM_L' \subset \sM_L$. Analogous results hold for the bottom row of the
diagram.]

Note that the characterization of $\sM$ given by
Theorem~\ref{th:odd-corners} enables us to check whether a
distribution given in the form
$$
T = \sum_{m,n=0}^\infty c_{mn} f_{mn} \in \sS'_2
$$
(whenever this series converges in $\sS'_2$) belongs to $\sM$ or not.
For example,
$$
\text{if}\quad T := \sum_{m,n=0}^\infty e^{-m} f_{mn},
\word{then} \|T\|_{st}^2
= \sum_{m=0}^\infty (2m+1)^s e^{-2m} \sum_{n=0}^\infty (2n+1)^t,
$$
which converges iff $t < -1$; thus $T \in \sM_L$ but $T \notin \sM_R$.
(Also, $T^* = \sum_{m,n=0}^\infty e^{-n} f_{mn}$ lies in $\sM_R$ but
not in $\sM_L$.) This shows that $\sM$, $\sM_L$, $\sM_R$ and~$\sS'_2$
are distinct.

We may also filtrate $\sS'_2$ by other types of Banach spaces, 
corresponding not to Hilbert--Schmidt operators but to trace-class or 
bounded operators on $L^2(\bR)$. Since $ZW$ is an isomorphism between 
$\sG_{0,0}$ and $\HS(L^2(\bR))$, we introduce
$$
\sI_{s,t} := \sG_{s,0} \x \sG_{0,t}
:= \set{f \x g : f \in \sG_{s,0},\,g \in \sG_{0,t}}.
$$
From \eqref{eq:eigen-basis} and \eqref{eq:many-norms} we find that
\begin{align*}
\sI_{s,t}
&= \set{H^{\x(-s/2)} \x f \x g \x H^{\x(-t/2)} : f,g \in \sG_{0,0}}
\\
&= \set{h \x k : h \in \sG_{s,q},\ k \in \sG_{-q,t}}
\quad\text{for any } q \in \bR.
\end{align*}
We have $h \in \sI_{0,0}$ iff $ZWh$ is a trace-class operator, and so
(by polar decomposition) we can find $u \in \sS'_2$ with
$u^* \x u = \one$ and $|h| \in \sI_{0,0}$ so that $h = u \x |h|$ and
$|h| = f^* \x f$ with $f \in \sG_{0,0}$. Writing
$\|h\|_{00,1} := \duo{\one}{|h|} = \|f\|_{00}^2$, we see that
$\sI_{0,0}$ is a Banach space and the inclusion
$\sI_{0,0} \subset \sG_{0,0}$ is continuous. Since,
by~\eqref{eq:many-norms},
$\|g\|_{st} = \|H^{\x(s/2)} \x g \x H^{\x(t/2)}\|_{00}$ for
$g \in \sG_{s,t}$, we may define
$$
\|f\|_{st,1} := \|H^{\x(s/2)} \x f \x H^{\x(t/2)}\|_{00,1}
$$
yielding the continuous inclusion $\sI_{s,t} \subset \sG_{s,t}$ for
all $s,t \in \bR$.

Next let $\sB_{s,t}$ be the dual space of $\sI_{-t,-s}$ under the 
duality $\scal{T}{f} := \half\duo{T^*}{f}$ for $T \in \sB_{s,t}$, 
$f \in \sI_{-t,-s}$. $\sB_{s,t}$ is naturally imbedded in $\sS'_2$. In
fact, we may identify $\sB_{s,t}$ with
$\{\, T \in \sS'_2 : \duo{T^*}{f}$ is finite for all 
$f \in \sI_{-t,-s} \,\}$. Note also~\cite{Kammerer86} that $\sB_{0,0}$
is the space of tempered distributions whose twisted product with any
function in $L^2(\bR^2)$ lies in $L^2(\bR^2)$. The norm
$$
\|T\|_\mathrm{op}
:= \sup\set{\|T \x f\|_{00}/\|f\|_{00} : f \in L^2(\bR^2),\ f \neq 0}
$$
on $\sB_{0,0}$ coincides with its norm as the
dual space of $\sI_{0,0}$ (so we have incidentally proved that the
dual of the space of trace-class operators on a separable Hilbert
space is the space of bounded operators).

Since $\sG_{s,t}$ and $\sG_{-t,-s}$ are dual Hilbert spaces under the
duality
$$
\biggl( \sum_{m,n=0}^\infty c_{mn}f_{mn} \biggm|
\sum_{k,l=0}^\infty d_{kl}f_{kl} \biggr)
:= \sum_{m,n=0}^\infty c_{mn}^* d_{mn}
$$
we obtain a chain of continuous inclusions
\begin{equation}
\sS_2 \subset \sI_{s,t} \subset \sG_{s,t} \subset \sB_{s,t}
\subset \sS'_2
\label{eq:short-chain} 
\end{equation}
where we define the norm of $\sB_{s,t}$ as the dual norm to 
$\|\cdot\|_{-t,-s,1}$. Denoting this norm by $\|\cdot\|_{st,\infty}$,
we obtain
\begin{equation}
\|f\|_{st,1} = \|H^{\x(s/2)} \x f\x H^{\x(t/2)}\|_{00,1}
\leq \|H^{\x(s/2)}\|_{00,1} \|f\|_{00,\infty} \|H^{\x(t/2)}\|_{00,1}
\label{eq:chain-estimate} 
\end{equation}
by the standard properties of trace-class operators, since 
$\|f\|_{00,\infty}$ is the norm of $ZWf$ as a bounded operator on
$L^2(\bR)$. So, since
$$
\|H^{\x(s/2)}\|_{00,1} = \duo{\one}{H^{\x(s/2)}} 
= \sum_{n=0}^\infty (2n + 1)^{s/2},
$$
which converges iff $s < -2$, we conclude 
from~\eqref{eq:chain-estimate} that $\sB_{0,0} \subset \sI_{s,t}$ if
$s < -2$, $t < -2$, and thus that $\sB_{p,q} \subset \sI_{s,t}$
(continuously) if $p > s + 2$, $q > t + 2$. Interpolating
$\sS_2 \subset \sB_{p,q} \subset \sI_{s,t}$ in~\eqref{eq:short-chain},
we see that we can replace $\sG_{s,t}$ by $\sI_{s,t}$ or $\sB_{s,t}$
in \eqref{eq:double-scale}. In effect, if we define
$$
\sI_{s,-\infty} := \bigcup_{t\in\bR} \sI_{s,t}\,,  \qquad
\sB_{s,-\infty} := \bigcup_{t\in\bR} \sB_{s,t}\,,
$$
we find that
$$
\sB_{p,-\infty} \subset \sI_{s,-\infty} \subset \sG_{s,-\infty} 
\subset \sB_{s,-\infty} \quad\text{if}\quad p > s + 2,
$$
and so
$$
\sM_L = \bigcap_{s\in\bR} \sG_{s,-\infty} 
= \bigcap_{s\in\bR} \sB_{s,-\infty}
= \bigcap_{s\in\bR} \sI_{s,-\infty}
$$
topologically, if all intersections have the natural (projective) 
topologies. In particular, the topology $\sT_3$ is induced by 
$\bigcap_{s\in\bR} \bigcup_{t\in\bR} \sB_{s,t}$.

\begin{thm} 
\label{th:idem-topologies}
The topologies $\sT_2$ and $\sT_3$ on $\sM_L$ coincide.
\end{thm}

\begin{proof}
Let us write
$$
(B;U) := \set{S \in \sM_L : S \x f \in U \text{ for all } f \in B},
$$
where $U$ is a neighbourhood of~$0$ in $\sS_2$ and $B$ is a bounded
set in $\sS_2$; then $(B;U)$ is a basic neighbourhood of~$0$
for~$\sT_2$.

A basic neighbourhood of~$0$ for $\sT_3$ is given by a subset 
$V \subset \sM_L$ such that $V = W \cap \sM_L$ where $W$ is a
neighbourhood of~$0$ in $\sG_{s,-\infty}$ for some $s \in \bR$. Then
$W \cap \sG_{s,-q}$ is a neighbourhood of~$0$ in $\sG_{s,-q}$ for all
$q \in \bR$ and so contains a set of the form
$\set{S \in \sG_{s,-q} : \|S\|_{s,-q} < C(q)}$ for all~$q$.

Given $(B;U)$, a neighbourhood of~$0$ for $\sT_2$, we may suppose that 
$U = \set{g \in \sS_2 : \|g\|_{st} \leq \dl}$, for some 
$s,t \in \bR$, $\dl > 0$, and that
$$
B = \set{f \in \sS_2
: \|f\|_{qp} \leq A(q,p) \text{ for all } q,p \in \bR},
$$
where $A(q,p) > 0$ is a suitable function. Set
$$
V := \bigcup_{q\in\bR} \set{S \in \sM_L \cap \sG_{s,-q}
: \|S\|_{s,-q} \leq \dl/A(q,t)}
$$
Then $V = W \cap \sM_L$, where $W \cap \sG_{s,-q}$ contains a ball of 
radius $\dl/A(q,t)$ for all $q$. Let $f \in B$, 
$S \in V \cap \sG_{s,-q}$; then we get
$$
\|S \x f\|_{st} \leq \|S\|_{s,-q} \|f\|_{qt}
\leq \biggl( \frac{\dl}{A(q,t)} \biggr) A(q,t) = \dl,
$$
so that $V \cap \sG_{s,-q} \subset (B;U)$ for all $q$, and thus 
$V \subset (B;U)$.

On the other hand, we observe that each Hilbert space $\sG_{s,t}$ 
is ``strictly webbed'' in the sense of de~Wilde~\cite{Koethe79}; this 
property is preserved under countable inductive limits (here 
$\sG_{s,-\infty} = \bigcup_{n\in\bN} \sG_{s,-n}$) and under countable 
projective limits; hence 
$(\sM_L,\sT_3) = \bigcap_{m\in\bN} \sG_{m,-\infty}$ is strictly
webbed.

Since $(\sM_L,\sT_1)$ is a Montel space~\cite{Grothendieck55} and
$\sT_1 = \sT_2$, $(\sM_L,\sT_2)$ is barrelled and complete:
furthermore, the identity map $\:(\sM_L,\sT_3) \to (\sM_L,\sT_2)$ is
continuous, by the first part of the proof. Now de~Wilde's open
mapping theorem~\cite{Koethe79} shows that this map is a
homeomorphism, and so $\sT_3 = \sT_2$.
\end{proof}

\begin{remk} 
As a corollary, we find that the space $\sL_b(\sS_1)$ is strictly
webbed.
\end{remk}

\begin{remk} 
We may summarize the topological properties of the Moyal algebra as
follows:
\begin{enumerate}
\item[(i)]
$\sM_L$ and~$\sM_R$ are complete, nuclear, reflexive locally convex
algebras with a hypocontinuous multiplication;
\item[(ii)]
$\sM$ is a complete, nuclear, semireflexive locally convex $*$-algebra
with a hypocontinuous multiplication and a continuous involution;
\item[(iii)]
using the technique outlined in Sec.~4 of~\textbf{I}, it is readily
seen that $\sM_L$, $\sM_R$ and~$\sM$ are Fourier-invariant normal
spaces of distributions.
\end{enumerate}
\end{remk}

\begin{remk} 
The technique of filtrating $\sS'_2$ by Hilbert spaces used here to
define $\sT_3$ may be employed to show that the dual space $\sM_L'$
can be represented as a dense ideal in $\sM_L$ (with a continuous
multiplication). Indeed, it can be shown that
$\sM_L' = \bigcup_{t\in\bR} \bigcap_{s\in\bR} \sG_{s,t}$ and that
$$
\sM_L' = \set{f \x T : f \in \sS_2,\ T \in \sS'_2} \subset \sO_T;
$$
it follows that $\sS_2 \subset \sM_L' \subset \sM_L$ and that $\sM_L'$
is an ideal. This we do in a following paper~\cite{Charon}.
\end{remk}

\section{Distributions corresponding to trace-class and bounded operators} 
\label{sec:trace-class}

The intermediate spaces $\sG_{s,t}$ and $\sB_{s,t}$ are useful for
several purposes. We may, for example, obtain information about
certain functions or tempered distributions by determining in which of
these spaces they lie. For instance, the identity $\one$ for the
twisted product lies in $\sB_{0,0}$ (of~course), but we may also
compute from~\eqref{eq:many-norms} that, since
$\one = \sum_{n=0}^\infty f_{nn}$, then
$$
\|\one\|_{st}^2 = \sum_{n=0}^\infty (2n+1)^{s+t}
$$
and hence $\one \in \sG_{s,t}$ iff $s + t < -1$.

The space $\sI_{0,0}$ corresponds to the trace-class operators on
$L^2(\bR)$. The question of which functions give rise to nuclear
operators, via the Weyl correspondence, has been studied by
Daubechies \cite{Daubechies83,Daubechies80}. She has identified, for
the present case of a two-dimensional phase space, a class of spaces
$\sW^r$ such that $\sW^r \subset \sI_{0,0}$ for $r > 1$, using a
coherent-state representation of Quantum Mechanics. From our point of
view, $\sW^r$ essentially consists of functions $f$ on~$\bR^2$ with
$\duo{f^*}{A^rf}$ finite, where $Af := H \x f + f \x H$. Since
$f \mapsto H \x f$ and $f \mapsto f \x H$ are commuting positive
operators (on $\sG_{0,0}$, say), we find that
$$
0 \leq \duo{f^*}{H^{\x r} \x f} \leq \duo{f^*}{A^rf}  \word{and}
0 \leq \duo{f^*}{f \x H^{\x r}} \leq \duo{f^*}{A^rf}
$$
for $r \geq 1$, and thus $\sW^r \subset \sG_{r,0} \cap \sG_{0,r}$. We
can now show an improved result.

\begin{thm} 
\label{th:trace-class}
$\sG_{r,0} \cup \sG_{0,r} \subset \sI_{0,0}$ if $r > 1$.
\end{thm}

\begin{proof}
We need only show that $\sG_{r,0} \subset \sI_{0,0}$. Take 
$f \in \sG_{r,0}$, and write
$$
f = \sum_{m,n=0}^\infty c_{mn}f_{mn} \word{with} 
\|f\|_{r0}^2 = \sum_{m,n=0}^\infty (2m+1)^r |c_{mn}|^2 \text{ finite}.
$$
Define $d_m$ by $d_m \geq 0$, 
$d_m^2 := (2m+1)^r \sum_{n=0}^\infty |c_{mn}|^2$. Then 
$\sum_{m=0}^\infty d_m^2 = \|f\|_{r0}^2$ so that 
$g := \sum_{m=0}^\infty d_m f_{mm}$ lies in~$\sG_{0,0}$. Define
$h := \sum_{m,n=0}^\infty b_{mn} f_{mn}$, where
$b_{mn} := c_{mn}/d_m$. We now observe that
$$
\sum_{m,n=0}^\infty |b_{mn}|^2 
= \sum_{m=0}^\infty \sum_{n=0}^\infty \frac{|c_{mn}|^2}{d_m^2} 
= \sum_{m=0}^\infty (2m + 1)^{-r} = (1 - 2^{-r}) \ze(r)
$$
so that $h \in \sG_{0,0}$ for $r > 1$. Thus $g \x h$ is defined and
lies in $\sI_{0,0}$, and it is clear that $f = g \x h$.

Furthermore, $\|g \x h\|_{00,1} \leq \|g\|_{00}\, \|h\|_{00}$ by a
well-known property of trace-class operators (transferred via the
isomorphism $ZW$ to the present context), so we get the estimate
$$
\|f\|_{00,1} \leq \bigl( (1 - 2^{-r}) \ze(r) \bigr)^{1/2} \|f\|_{r0}
$$
for $f \in \sG_{r,0}$. (Replacing $f$ by $f^*$, we see that the
analogous estimate is valid for $f \in \sG_{0,r}$.)
\end{proof}

\begin{remk} 
In the same vein, we observe that all distributions in 
$\sB_{0,0}$ lie in $\bigl( \bigcup_{r>1} \sG_{-r,0} \bigr)
\cap \bigl( \bigcup_{r>1} \sG_{0,-r} \bigr)$ with estimates:
\begin{align*}
\|T\|_{-r,0}
&\leq \bigl( (1 - 2^{-r}) \ze(r) \bigr)^{1/2} \|T\|_{00,\infty},
\\
\|T\|_{0,-r}
&\leq \bigl( (1 - 2^{-r}) \ze(r) \bigr)^{1/2} \|T\|_{00,\infty}.
\end{align*}
This is the tighest constraint of which we are aware on the class of
distributions corresponding to bounded operators by the Weyl rule. We
remark that our proofs are simpler than those of~\cite{Daubechies83}
since they merely involve manipulation of the double series introduced
in~\textbf{I}.
\end{remk}

\begin{remk} 
We have noted in~\textbf{I} that the twisted product of two square
integrable functions in $\bR^2$ lies in $C_0(\bR^2)$. Thus, if
$f \in \bigcup \set{\sG_{s,t} : s \geq 0,\ t \geq 0,\ s + t \geq 2}$,
then $f \in C_0(\bR^2)$. Then the Leibniz formula assures us that
$f \x g \in C_0^m(\bR^2)$ whenever
$$
f,g \in \bigcup \set{\sG_{s,t} : s > 2m,\ t > 2m,\ s + t > 4m + 2};
$$
analogously to what happens in the usual Sobolev spaces, the
distributions in $\sG_{s,t}$ grow more regular as $s$, $t$ become
larger in a suitable way.
\end{remk}

\section{Outlook} 

The formalism developed in~\textbf{I} and the present paper puts
forward a mathematical framework for phase-space quantum mechanics, in
which the usual calculus of unbounded operators is replaced by a
calculus of distributions on phase space, with some techniques of
locally convex space theory hovering in the background. Using this
framework, we have shown elsewhere~\cite{Triton} that the evolution
functions corresponding to any quadratic Hamiltonian on phase space
belong to the Moyal $*$-algebra $\sM$. To obtain more general results,
when the Schr\"odinger equation is not exactly solvable, we need an
appropriate spectral theorem for $\sM$ in order to apply semigroup
theory. As a step in this direction, we have identified the dual space
of~$\sM$ as a function space~\cite{Charon}. We hope to develop these
aspects further in a forthcoming paper.

\subsection*{Acknowledgements}

We are grateful for helpful correspondence in connection with the
present work from Seán Dineen, who brought Ref.~\cite{Vogt84} to our
attention, and also from John Horváth and Peter Wagner. We would
like to thank Prof.\ Abdus Salam and the International Centre for
Theoretical Physics, Trieste, for their hospitality during a stay in
which work was completed. We gratefully acknowledge support form the
Vicerrectoría de Investigación of the Universidad de Costa Rica.

\end{document}